\RequirePackage{fix-cm}
\documentclass[twocolumn]{svjour3}          
\smartqed  
\usepackage{graphicx}
\usepackage{epstopdf}
\usepackage{amssymb}
\usepackage{amsmath}
\usepackage{enumerate}
\usepackage[numbers]{natbib}
\usepackage{gensymb}
\usepackage{undertilde}
\usepackage[none]{hyphenat}
 \usepackage{mathptmx}      
\spnewtheorem{algorithm}{Algorithm}{\bf}{\rm}
\begin{document}

\title{On the piecewise-concave approximations of functions
}


\author{Gene A. Bunin 
}


\institute{ \email{gene.a.bunin@ccapprox.info}           
}

\date{Submitted: \today}

\maketitle

\begin{abstract}
The piecewise-concave function may be used to approximate a wide range of other functions to arbitrary precision over a bounded set. In this short paper, this property is proven for three function classes: (a) the multivariate twice continuously differentiable function, (b) the univariate Lipschitz-continuous function, and (c) the multivariate separable Lipschitz-continuous function.
\end{abstract}

Following Zangwill's definition \cite{Zangwill:67}, we define the piecewise-concave function, $p : \mathbb{R}^n \rightarrow \mathbb{R}$, as the pointwise maximum of $n_p$ concave functions $p_i$:

\vspace{-4mm}
\begin{equation}\label{eq:pwccv}
p(x) = \mathop {\max}_{i = 1,...,n_p} p_i(x),
\end{equation}

\noindent with $x \in \mathbb{R}^n$ the variable vector. While often arising directly in management science \cite{Zangwill:66,Zangwill:66b,Bhattacharjee:00,Chubanov:06} and location theory \cite{Giannessi:98} problems, the use of such functions as approximators of more general functions has been suggested more than once -- first by Zangwill himself \cite{Zangwill:67}, and then by Rozvany in the context of structural optimization \cite{Rozvany:70,Rozvany:71}. Recently, the piecewise-concave function has also been proposed as the link that allows the approximation of a  nonlinear programming problem by a reverse convex programming problem in nonconvex global optimization \cite{Bunin:ERCP}.

In the present paper, we examine the quality of the piecewise-concave approximation and prove that the approximation may be arbitrarily good for three general classes of functions over a bounded domain $\mathcal{X}$. These are:

\begin{enumerate}
\item the twice continuously differentiable ($\mathcal{C}^2$) function $f_c : \mathbb{R}^n \rightarrow \mathbb{R}$,
\item the Lipschitz-continuous univariate function $f_{u} : \mathbb{R} \rightarrow \mathbb{R}$,
\item the Lipschitz-continuous separable function $f_{s} : \mathbb{R}^n \rightarrow \mathbb{R}$. 
\end{enumerate}

\begin{theorem}[Piecewise-concave approximation of $\mathcal{C}^2$ functions] Let $f_c : \mathbb{R}^n \rightarrow \mathbb{R}$ be $\mathcal{C}^2$ over $\mathcal{X}$. It follows that there exists a piecewise-concave approximation $p$ such that

\vspace{-4mm}
\begin{equation}\label{eq:goodappC2}
\mathop {\max}_{x \in \mathcal{X}} | f_c(x) - p(x) | \leq \epsilon
\end{equation}

\noindent for any $\epsilon > 0$.
\end{theorem}
\begin{proof} The proof follows from the D.C. (difference of convex) decomposition of $f_c$ over $\mathcal{X}$ \cite[Corollary 4.1]{Horst1995}:

\vspace{-4mm}
\begin{equation}\label{eq:DCdecomp}
\begin{array}{l}
f_c(x) = f_{cvx} (x) + f_{ccv} (x) \\
f_{cvx} (x) = f_c(x) + \mu \|x\|_2^2 \\
f_{ccv} (x) = -\mu \| x \|_2^2,
\end{array}
\end{equation}

\noindent where the convexity of $f_{cvx}$ is assured for $\mu > 0$ sufficiently large. Since $f_{cvx}$ is clearly $\mathcal{C}^2$ over $\mathcal{X}$ as well, it follows that it can be approximated by a piecewise-linear function

\vspace{-4mm}
\begin{equation}\label{eq:pwlin}
l(x) = \mathop {\max}_{i = 1,...,n_p} \left( a_i^T x + b_i  \right)
\end{equation}

\noindent such that

\vspace{-4mm}
\begin{equation}\label{eq:goodapplin}
\mathop {\max}_{x \in \mathcal{X}} | f_{cvx}(x) - l(x) | \leq \epsilon
\end{equation}

\noindent for any $\epsilon > 0$. Choosing

\vspace{-4mm}
\begin{equation}\label{eq:goodapplin2}
p(x) = f_{ccv}(x) + l(x) = \mathop {\max}_{i = 1,...,n_p} \left( f_{ccv}(x) + a_i^T x + b_i  \right)
\end{equation}

\noindent and reformulating (\ref{eq:goodapplin}) yields the desired result:

\vspace{-4mm}
\begin{equation}\label{eq:goodapplin3}
\begin{array}{l}
\displaystyle \mathop {\max}_{x \in \mathcal{X}} | f_{cvx}(x) + f_{ccv}(x) - f_{ccv}(x)  - l(x) |  = \vspace{1mm} \\
\displaystyle \hspace{35mm} \mathop {\max}_{x \in \mathcal{X}} | f_c (x) - p(x) | \leq \epsilon. \;\;\; \qed
\end{array}
\end{equation}

\end{proof}

From the point of view of actually computing the approximation, the above result is largely conceptual in nature since a D.C. decomposition may not be available for a given $f_c$, and one has to have a lower bound on the minimum eigenvalue of the Hessian of $f_c$ to know what value of $\mu$ is ``sufficiently large'' \cite{Adjiman:96}. In the case where a D.C. decomposition is available, obtaining the approximation simply becomes a matter of approximating $f_{cvx}$, for which very simple methods such as discretizing and taking linear approximations of $f_{cvx}$ at the discretization points could suffice.

\begin{theorem}[Approximation of Lipschitz-continuous univariate functions] Let $f_u : \mathbb{R} \rightarrow \mathbb{R}$ be Lipschitz- continuous over $\mathcal{X} = \{ x \in \mathbb{R} : \underline x \leq x \leq \overline x \}$:

\vspace{-4mm}
\begin{equation}\label{eq:lipcont}
| f_u (x_a) - f_u (x_b) | < \kappa | x_a - x_b |, \;\; \forall x_a, x_b \in \mathcal{X} \;\; (x_a \neq x_b),
\end{equation}

\noindent with $\kappa > 0$ denoting the Lipschitz constant. It follows that there exists a piecewise-concave approximation $p$ such that

\vspace{-4mm}
\begin{equation}\label{eq:goodappuni}
\mathop {\max}_{x \in \mathcal{X}} | f_u(x) - p(x) | \leq \epsilon
\end{equation}

\noindent for any $\epsilon > 0$.

\end{theorem}
\begin{proof} Let $p$ be defined by concave parabolas:

\vspace{-4mm}
\begin{equation}\label{eq:pchoice}
p(x) = \mathop {\max} \limits_{i = 1,...,n_p} \left( \beta_{2,i} x^2 + \beta_{1,i} x + \beta_{0,i} \right),
\end{equation}

\noindent where $\beta_{2} \in \mathbb{R}^{n_p}_{-}$ and $\beta_{1}, \beta_0 \in \mathbb{R}^{n_p}$, and consider the discretization given by $x_d = \{ \underline x, \underline x + \Delta x, ..., \overline x - \Delta x, \overline x \}$, with $\Delta x > 0$ dictating the precision. Let $n_p = (\overline x - \underline x)/\Delta x$ be the number of discretization subintervals, each of length $\Delta x$. 

We will enforce that each $p_i(x) = \beta_{2,i} x^2 + \beta_{1,i} x + \beta_{0,i}$ satisfy the following criteria:

\vspace{-4mm}
\begin{equation}\label{eq:pwquad3}
\begin{array}{l}
p_i(x_{d,i} + 0.5 \Delta x) = f_u(x_{d,i} + 0.5 \Delta x) \\
\displaystyle \frac {d p_i}{dx} \Big | _{x_{d,i}} = 2 \kappa \\
\displaystyle \frac {d p_i}{dx} \Big | _{x_{d,i+1}} = -2 \kappa,
\end{array}
\end{equation}

\noindent where $x_{d,i}$ denotes the $i^{\rm th}$ element of $x_d$. If written and solved as a linear system, (\ref{eq:pwquad3}) translates into the following:

\vspace{-4mm}
\begin{equation}\label{eq:pwquad4}
\begin{array}{l}
\left[ {\begin{array}{*{20}c}
   \beta_{2,i}  \\
   \beta_{1,i} \\
   \beta_{0,i}
\end{array}} \right] = 
\left[ {\begin{array}{*{20}c}
   (x_{d,i} + 0.5 \Delta x)^2 & x_{d,i} + 0.5 \Delta x & 1  \\
   2x_{d,i} & 1 & 0 \\
   2x_{d,i+1} & 1 & 0
\end{array}} \right] ^{-1} \vspace{1mm} \\ \hspace{45mm} 
\left[ {\begin{array}{*{20}c}
   f_u(x_{d,i} + 0.5 \Delta x)  \\
   2 \kappa \\
   -2 \kappa
\end{array}} \right].
\end{array}
\end{equation}

This solution exists and is unique as long as $\Delta x > 0$, with the resulting $p_i$ expressed analytically as

\vspace{-4mm}
\begin{equation}\label{eq:piana}
\begin{array}{ll}
 p_i (x) = & \displaystyle -\frac{2 \kappa}{\Delta x} x^2 - 2\kappa \left( 1-\frac{2 x_{d,i+1}}{\Delta x} \right) x - \frac{2 \kappa x_{d,i}^2}{\Delta x} \vspace{1mm} \\
& - 0.5\kappa \Delta x - 2\kappa x_{d,i} + f_u(x_{d,i} + 0.5 \Delta x).
\end{array}
\end{equation}

By enforcing the three conditions of (\ref{eq:pwquad3}), the following properties are guaranteed:

\begin{enumerate}
\item $p_i$ is quadratic and concave, with $\beta_{2,i} = -2 \kappa / \Delta x < 0$.
\item $p_i$ is a strict underestimator of $f_u$ at all points in $[\underline x, \overline x]$ that are outside the open interval $(x_{d,i}, x_{d,i+1})$. This may be proven as follows.

First, consider the function

\vspace{-4mm}
\begin{equation}\label{eq:lipsaw}
L_i (x) = f_u (x_{d,i} + 0.5 \Delta x) - \kappa | x - x_{d,i} - 0.5 \Delta x |,
\end{equation}

\noindent which is the Lipschitz ``sawtooth'' underestimator of $f_u$, generated around $x = x_{d,i} + 0.5 \Delta x$. It follows from the definition of the Lipschitz constant that

\vspace{-4mm}
\begin{equation}\label{eq:lipbound}
L_i (x) < f_u (x), \;\; \forall x \in [\underline x, \overline x] \setminus \{ x_{d,i} + 0.5 \Delta x \}.
\end{equation}

Given the construction of $p_i$, one sees that $L_i (x) = p_i(x)$ at $x = x_{d,i}, x_{d,i+1}$. Consider now the function 

\vspace{-4mm}
\begin{equation}\label{eq:pilin}
\overline p_i (x) = 2 \kappa x + f_u(x_{d,i} + 0.5 \Delta x) - 2\kappa x_{d,i} - 0.5 \kappa \Delta x,
\end{equation}

\noindent which is the linearization of $p_i$ at $x = x_{d,i}$. It is evident that $\overline p_i (x) \leq L_i (x), \; \forall x \in [\underline x, x_{d,i}]$, as both are linear and intersect at $x_{d,i}$, with $\overline p_i$ having a greater positive slope. From the concavity of $p_i$, it is also true that $p_i (x) \leq \overline p_i (x), \; \forall x$. It follows that

\vspace{-4mm}
\begin{equation}\label{eq:pwquad5}
p_i (x) \leq \overline p_i (x) \leq L_i (x) < f_u(x), \; \forall x \in [\underline x, x_{d,i}].
\end{equation}

A symmetrical analysis around $x_{d,i+1}$ yields a symmetrical result, and combining the two yields

\vspace{-4mm}
\begin{equation}\label{eq:pwquad6}
p_i (x) < f_u(x), \; \forall x \in [\underline x, x_{d,i}] \cup [x_{d,i+1}, \overline x ].
\end{equation}

\item $p_i$ approximates $f_u$ with zero error at $x = x_{d,i} + 0.5 \Delta x$. 
\item The interval for which $p_i(x) = p(x)$ is a strict subinterval of $[x_{d,i} - 0.5 \Delta x, x_{d,i+1} + 0.5\Delta x]$, i.e., $p_i$ can only be the ``piece'' of the piecewise-maximum function in the interior of this interval. This may be proven as follows.

Supposing first that $1 < i < n_p$, let $p_{i-1}$ denote the concave quadratic function for the neighboring interval $[x_{d,i-1},x_{d,i}]$, and consider the difference

\vspace{-4mm}
\begin{equation}\label{eq:pwquad7}
\begin{array}{l}
p_{i-1} (x) - p_i (x) = -4\kappa (x - x_{d,i}) \vspace{1mm} \\
\hspace{15mm} + f_u(x_{d,i} - 0.5 \Delta x) - f_u(x_{d,i} + 0.5 \Delta x).
\end{array}
\end{equation}

For $x = x_{d,i} - 0.5 \Delta x$, one may build on the result of (\ref{eq:pwquad6}), which states that $p_i (x_{d,i} - 0.5 \Delta x) < f_u(x_{d,i} - 0.5 \Delta x)$, and Property 3, which states that $p_{i-1} (x_{d,i} - 0.5 \Delta x) = f_u(x_{d,i} - 0.5 \Delta x)$, to obtain the following:

\vspace{-4mm}
\begin{equation}\label{eq:pwquad8}
\begin{array}{l}
-p_i(x_{d,i} - 0.5 \Delta x) > -f_u(x_{d,i} - 0.5 \Delta x) \\
p_{i-1} (x_{d,i} - 0.5 \Delta x) = f_u(x_{d,i} - 0.5 \Delta x) \\
\Rightarrow p_{i-1} (x_{d,i} - 0.5 \Delta x) - p_i (x_{d,i} - 0.5 \Delta x) > 0,
\end{array}
\end{equation}

which shows that the piece $p_{i-1}$ must be greater than $p_{i}$ at $x = x_{d,i} - 0.5 \Delta x$. From examining (\ref{eq:pwquad7}), it is clear that the derivative of this difference with respect to $x$ is negative, i.e., the difference increases with decreasing $x$. This implies that $p_{i-1} (x) - p_i (x) > 0$ remains true on the interval $x \in [\underline x, x_{d,i} - 0.5 \Delta x]$, and that $p_i$ cannot be the maximal piece on this interval. A symmetrical analysis shows that $p_{i+1}(x) - p_i(x) > 0$ for $x \in [x_{d,i+1} + 0.5 \Delta x, \overline x]$, i.e., that $p_i$ cannot be the maximal piece on this interval either. The overall result is thus summarized as

\vspace{-4mm}
\begin{equation}\label{eq:pwquad9}
\begin{array}{l}
p_i (x) < p (x), \vspace{1mm} \\
\hspace{10mm} \forall x \in [\underline x, x_{d,i} - 0.5 \Delta x] \cup [x_{d,i+1} + 0.5 \Delta x, \overline x]. 
\end{array}
\end{equation}

\noindent For the edge cases of $p_1$ and $p_{n_p}$, the same analysis applies but only one side has to be considered for each, since the other falls outside of $[\underline x, \overline x]$. In particular, the results obtained for the edge cases would be as follows:

\vspace{-4mm}
\begin{equation}\label{eq:pwquad9a}
\begin{array}{l}
p_1 (x) < p (x), \;\; \forall x \in [\underline x + 1.5 \Delta x, \overline x] \vspace{1mm} \\
p_{n_p} (x) < p (x), \;\; \forall x \in [\underline x, \overline x - 1.5 \Delta x].
\end{array}
\end{equation}

\end{enumerate}

Together, Properties 2 and 3 imply that $p(x) = f_u (x)$ at the midpoint of each discretization interval $[x_{d,i},x_{d,i+1}]$, with Property 3 establishing the zero-error approximation due to the piece $p_i$ and Property 2 establishing that every other piece must strictly underestimate the function at this point.

It now remains to consider the approximation error between the midpoints of the discretization intervals, for which the first step requires the identification of the Lipschitz constant of $p$. By Property 4, every piece $p_i$ is limited to the open interval $(x_{d,i} - 0.5 \Delta x, x_{d,i+1} + 0.5 \Delta x)$, from which it follows that the Lipschitz constant of $p$ cannot exceed the Lipschitz constant of one of these pieces over the relevant interval:

\vspace{-4mm}
\begin{equation}\label{eq:lippi}
\begin{array}{l}
\displaystyle \mathop {\sup} \limits_{x \in \left( {\footnotesize \begin{array}{l}  x_{d,i} - 0.5 \Delta x, \\ x_{d,i+1} + 0.5 \Delta x  \end{array} } \right) } \Bigg | \frac{dp_i}{dx} \Big |_x  \Bigg | \\
\displaystyle \hspace{5mm} = \mathop {\sup} \limits_{x \in \left( {\footnotesize \begin{array}{l}  x_{d,i} - 0.5 \Delta x, \\ x_{d,i+1} + 0.5 \Delta x  \end{array} } \right) } \Big | \frac{4\kappa}{\Delta x} (x_{d,i} - x) - 2\kappa \Big | = 4 \kappa.
\end{array}
\end{equation}

\noindent This allows for the approximation error to be bounded with respect to any discretization interval midpoint $x_{d,i} + 0.5 \Delta x$ by considering the Lipschitz sawtooth bounds for both $f_u$ and $p$:

\vspace{-4mm}
\begin{equation}\label{eq:pwquad10}
\begin{array}{l}
f_u(x_{d,i} + 0.5 \Delta x) - \kappa | x - x_{d,i} - 0.5 \Delta x |  \vspace{1mm} \\
\hspace{5mm}\leq f_u(x) \leq f_u(x_{d,i} + 0.5 \Delta x) + \kappa | x - x_{d,i} - 0.5 \Delta x |, \vspace{3mm} \\
p(x_{d,i} + 0.5 \Delta x) - 4\kappa | x - x_{d,i} - 0.5 \Delta x |  \vspace{1mm} \\
\hspace{5mm}\leq p(x) \leq p(x_{d,i} + 0.5 \Delta x) + 4\kappa | x - x_{d,i} - 0.5 \Delta x |,
\end{array}
\end{equation}

\noindent $\forall x \in [\underline x, \overline x]$. Negating the latter:

\vspace{-4mm}
\begin{equation}\label{eq:pwquad10a}
\begin{array}{l}
-p(x_{d,i} + 0.5 \Delta x) - 4\kappa | x - x_{d,i} - 0.5 \Delta x | \vspace{1mm}\\
\hspace{5mm} \leq -p(x) \leq -p(x_{d,i} + 0.5 \Delta x) + 4\kappa | x - x_{d,i} - \Delta x | \\
\end{array}
\end{equation}

\noindent and adding it to the former, while noting that $f_u(x_{d,i} + 0.5 \Delta x) = p(x_{d,i} + 0.5 \Delta x)$, yields

\vspace{-4mm}
\begin{equation}\label{eq:pwquad10b}
\begin{array}{l}
- 5\kappa | x - x_{d,i} - 0.5 \Delta x | \\
\hspace{10mm} \leq f_u(x) - p(x) \leq 5\kappa | x - x_{d,i} - 0.5 \Delta x |, 
\end{array}
\end{equation}

\noindent which is equivalent to

\vspace{-4mm}
\begin{equation}\label{eq:pwquad10c}
| f_u(x) - p(x) | \leq 5\kappa | x - x_{d,i} - 0.5 \Delta x |, \;\; \forall x \in [\underline x, \overline x]. 
\end{equation}

Without loss of generality, we may suppose $x$ to lie between the discretization points $x_{d,i}$ and $x_{d,i+1}$, i.e., that

\vspace{-4mm}
\begin{equation}\label{eq:midpoint}
x = \theta x_{d,i} + (1-\theta)x_{d,i+1}, \; \theta \in [0,1].
\end{equation}

\noindent Since $x_{d,i+1} = x_{d,i} + \Delta x$, this may be rewritten as

\vspace{-4mm}
\begin{equation}\label{eq:midpoint}
x = \theta x_{d,i} + (1-\theta)(x_{d,i} + \Delta x) = x_{d,i} + \Delta x - \theta \Delta x,
\end{equation}

\noindent and substituted into (\ref{eq:pwquad10c}) to obtain

\vspace{-4mm}
\begin{equation}\label{eq:pwquad10d}
| f_u(x) - p(x) | \leq 5\kappa \Delta x | 0.5 - \theta |, \;\; \forall x \in [\underline x, \overline x]. 
\end{equation}

Given that $\theta$ must lie in the unit interval, the worst-case upper bound that is independent of $\theta$ clearly corresponds to the cases where $\theta$ is either 0 or 1, and as such

\vspace{-4mm}
\begin{equation}\label{eq:pwquad10e}
| f_u(x) - p(x) | \leq 2.5\kappa \Delta x, \;\; \forall x \in [\underline x, \overline x]. 
\end{equation}

For a given $\epsilon$, it then suffices to choose $\Delta x = \frac{\epsilon}{2.5 \kappa}$ to obtain the desired result. \qed
\end{proof}

In this case, we note that the proof provides us with a simple method to construct a piecewise-concave approximation to arbitrary precision, provided that a proper estimate of the Lipschitz constant $\kappa$ is available. For a univariate function on a bounded interval, it is expected that obtaining such an estimate should not be very difficult for most problems.

The approximation result for a Lipschitz-continuous separable function follows as a corollary to Theorem 2.

\begin{corollary}[Approximation of a Lipschitz-continuous separable function] Let $f_s : \mathbb{R}^n \rightarrow \mathbb{R}$ be Lipschitz- continuous and separable over $\mathcal{X}$:

\vspace{-4mm}
\begin{equation}\label{eq:sep}
f_s (x) = \sum_{j=1}^{n} f_{u,j} (x_j),
\end{equation}

\noindent with $f_{u,j} : \mathbb{R} \rightarrow \mathbb{R}$ denoting its univariate components. It follows that there exists a piecewise-concave approximation $p$ such that:

\vspace{-4mm}
\begin{equation}\label{eq:goodappsep}
\mathop {\max}_{x \in \mathcal{X}} | f_s(x) - p(x) | \leq \epsilon
\end{equation}

\noindent for any $\epsilon > 0$.

\end{corollary}
\begin{proof} The Lipschitz continuity of $f_s$ implies the Lipschitz continuity of its univariate components $f_{u,j}$. Likewise, the boundedness of $\mathcal{X}$ implies that the individual variables $x_j$ may be bounded by some finite $\underline x_j, \overline x_j$ so that $\underline x_j \leq x_j \leq \overline x_j, \; \forall j = 1,...,n$. It then follows from Theorem 2 that for each $j$ there exists a piecewise-concave approximation $p_j : \mathbb{R} \rightarrow \mathbb{R}$ such that

\vspace{-4mm}
\begin{equation}\label{eq:goodappunij}
\mathop {\max}_{x_j \in [\underline x_j, \overline x_j ]} | f_{u,j}(x_j) - p_j(x_j) | \leq \epsilon_j
\end{equation}

\noindent for any $\epsilon_j > 0$.

An equivalent statement to (\ref{eq:goodappunij}) is that

\vspace{-4mm}
\begin{equation}\label{eq:goodappunij2}
-\epsilon_j \leq f_{u,j}(x_j) - p_j(x_j)  \leq \epsilon_j, \;\; \forall x_j \in [\underline x_j, \overline x_j],
\end{equation}

\noindent which, if summed over $j = 1,...,n$, yields

\vspace{-4mm}
\begin{equation}\label{eq:goodappunij3}
- \sum_{j=1}^{n} \epsilon_j \leq \sum_{j=1}^n f_{u,j}(x_j) - \sum_{j=1}^n p_j(x_j)  \leq \sum_{j=1}^n \epsilon_j, \;\; \forall x \in \mathcal{X},
\end{equation}

\noindent or

\vspace{-4mm}
\begin{equation}\label{eq:goodappunij4}
- \sum_{j=1}^{n} \epsilon_j \leq f_{s}(x) - \sum_{j=1}^n p_j(x_j)  \leq \sum_{j=1}^n \epsilon_j, \;\; \forall x \in \mathcal{X}.
\end{equation}

Let us choose

\vspace{-4mm}
\begin{equation}\label{eq:pwsum}
p(x) = \sum_{j=1}^n p_j(x_j),
\end{equation}

\noindent which must be piecewise-concave since the sum of continuous piecewise-concave functions must also be continuous piecewise-concave \cite{Zangwill:67}. Substituting (\ref{eq:pwsum}) into (\ref{eq:goodappunij4}) and returning to the equivalent worst-case formulation yields:

\vspace{-4mm}
\begin{equation}\label{eq:goodappunij5}
\mathop {\max}_{x \in \mathcal{X}} | f_{s}(x) - p(x) | \leq \sum_{j=1}^n \epsilon_j,
\end{equation}

\noindent where choosing, as one example, $\epsilon_j = \epsilon/n$ yields the desired result. \qed

\end{proof}

\bibliographystyle{spmpsci}      
\bibliography{pwccv}
%
%



\end{document}